\documentclass[12pt]{amsart}
\usepackage{amsfonts,amscd,latexsym,amsmath,amssymb,cancel,enumerate}
\theoremstyle{plain}
\newtheorem{master}{Master}[section]
\newtheorem{prop}[master]{Proposition}
\newtheorem{thm}[master]{Theorem}
\newtheorem{fact}[master]{Fact}
\newtheorem{lem}[master]{Lemma}

\newtheorem{question}[master]{Question}
\newtheorem{claim}[master]{Claim}
\theoremstyle{definition}
\newtheorem{defin}[master]{Definition}
\theoremstyle{remark}
\newtheorem{remark}[master]{Remark}
\numberwithin{equation}{section}

\newcommand{\Cond}[1]{(#1)_A^\omega}
\newcommand{\Condi}[2]{(#1)_A^{#2}}
\newcommand{\RP}[1]{\Cond{\omega}\times_{R#1} \Cond{\omega}}
\newcommand{\LP}[1]{\leq_{R#1}}

\begin{document}
\title[Canonization for the Carlson-Simpson forcing]{Canonization of analytic equivalences on the Carlson-Simpson forcing}
\author{Michal Doucha}
\address{Institute of Mathematics, Academy of Sciences, Prague, Czech republic}
\email{m.doucha@post.cz}
\thanks{The research of the author was partially supported by grant IAA100190902 of Grant Agency of the Academy of Sciences of the Czech Republic}
\begin{abstract}
We prove a canonization result for the Carlson-Simpson forcing in the spirit of \cite{KSZ}. We generalize the weak form of the Carlson-Simpson theorem (\cite{CaSi}) dealing with partitions without free blocks: instead of dealing with finite Borel (resp. Baire-property) colorings we deal with (uncountable) colorings such that the corresponding equivalence relation (two partitions are equivalent if they are colored by the same color) is analytic.
\end{abstract}
\maketitle
\section*{Introduction}
In \cite{CaSi}, Timothy J. Carlson and Stephen G. Simpson prove a strong combinatorial theorem concerning finite partitions of natural numbers that is in some sense dual to the classical Ramsey theorem. It is usually called the Dual Ramsey theorem or the Carlson-Simpson theorem. In this paper we connect this combinatorial result with the research program from \cite{Zap} and \cite{KSZ}. We define a forcing notion, resp. a $\sigma$-ideal on a certain Polish space, that corresponds to the object studied in the Dual Ramsey theorem and prove a canonization result for this $\sigma$-ideal. More specifically, we identify a finite set of equivalence relations that are in the spectrum of this ideal and any other analytic equivalence relation canonizes to one of them. This result can be viewed as a generalization of the weaker form of the Carlson-Simpson theorem dealing with partitions of $\omega$ without free blocks (Lemma 2.3 in \cite{CaSi}).

Before defining necessary notions, we can state one immediate interesting consequence of our result.
\begin{thm}\label{introcor}
Let $E$ be any analytic equivalence relation on $\mathcal{P}(\omega)$ (we identify elements of $\mathcal{P}(\omega)$ with elements of $2^\omega$). Then there exists an infinite sequence $(A_n)_{n\in \omega}$ of pairwise disjoint non-empty subsets of $\omega$ (finite or infinite) such that either for any two different arbitrary unions of such sets (both containing $A_0$ though) they are $E$-equivalent, or for any two different arbitrary unions both containing $A_0$ they are $E$-inequivalent.
\end{thm}
\section{Preliminaries and basic notions}
In order to state and motivate our results we present the general program of \cite{KSZ}: let $X$ be a Polish space, $I$ a $\sigma$-ideal on $X$ and $E\subseteq X^2$ an analytic equivalence relation.
\begin{itemize}
\item We say that $E$ is in the spectrum of $I$ if there exists a Borel set $B\in I^+$ such that $\forall C\in (I^+\cap \mathrm{Borel}(B))$ $E\upharpoonright C$ has the same complexity as $E$ on the whole space, i.e. $E\upharpoonright C$ is Borel bireducible with $E\upharpoonright X$.
\item On the other hand, $I$ canonizes $E$ to a relation $F\leq _\mathrm{B} E$ if for every Borel $B\in I^+$ there is some Borel $C\in (I^+\cap \mathrm{Borel}(B))$ such that $E\upharpoonright C$ is bireducible with $F$.
\end{itemize}
  
Now we introduce the original notation of Carlson and Simpson from \cite{CaSi} and state their theorem. Then we define the forcing notion, resp. the $\sigma$-ideal and state our result.
\begin{defin}
Let $A$ be a finite (at least two-element) alphabet. As in \cite{CaSi}, by $\Condi{\omega}{\alpha}$, where $\alpha\in (\omega\setminus |A|)\cup\{\omega\}$, we denote the set of all partitions of $A\cup \omega$ into $\alpha$ pieces such that two different elements $a\neq b\in A$ lie in two different pieces of such partitions. For any $X\in \Condi{\omega}{\alpha}$, a piece containing some $a\in A$ is called an $a$-block, a piece not containing any element of $A$ is called a free block.

For $Y\in \Condi{\omega}{\beta}$ and $X\in \Condi{\omega}{\alpha}$, where $\beta\leq \alpha$, we say $Y$ is coarser than $X$, $Y\preceq X$, if every block of $X$ is contained in some block of $Y$. For any $X\in \Condi{\omega}{\alpha}$ by $\Condi{X}{\beta}$, where $\beta\leq \alpha$, we denote the set $\{Y\in \Condi{\omega}{\beta}:Y\preceq X\}$.
\end{defin}
\begin{defin}[Space $\Condi{\omega}{0}$]
Let $A$ be as before. Consider the set $\Condi{\omega}{0}$. We look at it as a set of all partitions of $\omega$ into $|A|$ pieces indexed by $A$. There is a natural correspondence between $\Condi{\omega}{0}$ and $A^\omega$. The latter carries a product topology if we consider $A$ as a discrete space which is homeomorphic to the topology of the Cantor space. From now on we will not distinguish between these two sets and thus be able to speak about topological properties of $\Condi{\omega}{0}$.
\end{defin}
\begin{defin}[Carlson-Simpson forcing/ideal]
We shall consider $(\Cond{\omega},\preceq)$ as a forcing notion. For $X\in \Cond{\omega}$ we shall write $[X]$ to denote the set $(X)_A^0$. Note that for any such $X$, $[X]$ is a closed subset of $\Condi{\omega}{0}$ (or $A^\omega$).

Let $I_{C_n}\subseteq \mathcal{P}(A^\omega)$, where $n$ denotes the cardinality of $A$, be the set of all Borel subsets of $A^\omega$ that do not contain $[X]$ for some $X\in \Cond{\omega}$.
\end{defin}
The following proposition gives some properties of $I_{C_n}$.
\begin{prop}\emph{}\label{CSprop}
\begin{enumerate}
\item $I_{C_n}$ is a $\sigma$-ideal.
\item  $P_{I_{C_n}}$ is forcing equivalent to $(\Cond{\omega},\preceq)$.
\item  $P_{I_{C_n}}$ is proper.
\end{enumerate}
\end{prop}
We postpone the proof until we have proved the main theorem \ref{CScan}. The reason for that is that the first item of Proposition \ref{CSprop} will follow easily. We do not need any part of the proposition in the proof of the main theorem. However, let us mention that all items of Proposition \ref{CSprop} could be proved by a direct argument without applying the main theorem.

Let us state a restricted version of the Carlson-Simpson (Dual Ramsey) theorem for partitions without free blocks.
\begin{thm}[Carlson-Simpson \cite{CaSi}, Lemma 2.3]\label{CSThm}
For any $X\in \Cond{\omega}$ and any finite partition $[X]=C_0\cup\ldots \cup C_n$ into pieces having the Baire property there exists $Y\in \Cond{X}$ and $i\leq n$ such that $[Y]\subseteq C_i$.
\end{thm}
We conclude this section by the following fact that will serve as a useful tool in the proof of the main theorem by forcing.
\begin{fact}[Analytic absoluteness; see \cite{Je} Theorem 25.4]\label{analabs}
Let $a\in \omega^\omega$ be a parameter and $\phi$ be a $\Sigma_1^1$ formula with free variables. Then $\phi(a)$ is absolute for all transitive models of a large enough fragment of set theory containing the parameter $a$.
\end{fact}
In particular, if $A\subseteq X$ is an analytic subset of a standard Borel space $X$, $x\in X$ is arbitrary and $M$ and $N$ are two transitive models of set theory containing $A, X$ (resp. codes for them) and $x$, then $M\models x\in A$ iff $N\models x\in A$.
\section{Canonization}
Let $A$ be a finite alphabet such that $|A|=n\geq 2$. Let $\mathcal{B}$ be a partition of $A$. Then we can consider the following equivalence relation $E_\mathcal{B}$ on $A^\omega$: for $x,y\in A^\omega$ we set $x E_\mathcal{B} y$ iff $\forall n\in \omega\forall B\in \mathcal{B} (x(n)\in B\Leftrightarrow y(n)\in B)$.

It is easy to check that $E_B$ is a closed equivalence relation that is in the spectrum of $I_{C_n}$. For a finite alphabet $A$ let $\mathcal{P}_A$ denote the set of partitions of $A$. The main result says that these are the only analytic equivalences in the spectrum of $I_{C_n}$. Every other analytic equivalence relation canonizes to one of them.
\begin{thm}\label{CScan}
Fix some finite alphabet $A$ with at least two elements. Let $X\in \Cond{\omega}$ be a condition in the Carlson-Simpson forcing and $E$ an analytic equivalence relation on $[X]$ (i.e. an analytic subset of $[X]^2$). Then there exists a subcondition $Y\in  \Cond{X}$ such that $E\upharpoonright [Y]$ is equal to $[Y]\times [Y]$ or to $\mathrm{id}([Y])$ or there exists $\mathcal{B}\in \mathcal{P}_A$ such that $E\upharpoonright [Y]=E_\mathcal{B}\upharpoonright [Y]$.

In particular, we have a total canonization for $I_{C_2}$.
\end{thm}
Of course, $\mathrm{id}([Y])=E_\mathcal{B}\upharpoonright [Y]$ for $\mathcal{B}$ a partition into singletons, and $[Y]\times [Y]=E_\mathcal{B}\upharpoonright [Y]$ for $\mathcal{B}=\{A\}$.
\begin{remark}
This is an ``almost generalization" of Theorem \ref{CSThm} as this theorem can be viewed as a canonization result for equivalence relations having finitely many classes. We used the term ``almost generalization" as the Theorem \ref{CSThm} holds for partitions into pieces having the Baire property whereas Theorem \ref{CScan} generalizes only the case with analytic partitions.
\end{remark}
\begin{defin}
Let $A$ be from the statement of the theorem. For any $X\in \Condi{\omega}{\alpha}$ and any $a\in A$, by $X(a)$ we denote the elements of $\omega$ belonging to the $a$-block; i.e. $\{a\}\cup X(a)$ is the $a$-block of $X$. Similarly, for any $n<\alpha$ by $X(n)$ we denote the elements of $\omega$ belonging to the $n$-th free block where the enumeration of free blocks is determined by their minimal elements. 
\end{defin}
Thus for instance, if $X$ has two free blocks, $\{0,1,3,4,\ldots\}$ and $\{2\}$, the the $0$-th block is $\{0,1,3,4,\ldots\}$ and the $1$-st block is  $\{2\}$.
\begin{defin}
Let us write $A=\{a_0,\ldots,a_{n-1}\}$ and let $X\in \Cond{\omega}$ be a condition and $s\in n^{<\omega}$ a finite $n$-ary sequence. By $X^s$ we denote the condition $Y\in \Cond{X}$ for which for every $i<n$ $Y(a_i)=X(a_i)\cup \bigcup \{X(k): k<|s| \wedge s(k)=i\}$ and for every $m\in \omega$ $Y(m)=X(m+|s|)$.
\end{defin}
Note that whenever for some $X$ and $s$ there is $Y\in \Cond{X^s}$, then there is in fact a condition $Z$ such that $Z\in \Cond{X}$ and $Z^s=Y$; i.e. $Z(a_i)=Y(a_i)\setminus \bigcup_{i<|s|} X(i)$ for $i<n$, $Z(i)=X(i)$ for $i<|s|$ and $Z(i)=Y(i-|s|)$ for $i\geq |s|$. From that reason for a condition $X$ and a finite $n$-ary sequence $s$ when we write $Z^s\leq X^s$, then by $Z$ we mean the condition ($\in \Cond{X}$) described above.

We would like to use fusion of conditions, so in the next definition we define what fusion sequence is.
\begin{defin}[Fusion sequence]
We define the order $\preceq_m\subseteq \preceq$ for every $m$. For $X,Y\in \Cond{\omega}$ and $m\in \omega$ $Y\preceq_m X$ if $Y\in \Cond{X}$ and $\forall j<m (Y(j)\supseteq X(j))$. In particular, $Y\preceq_0 X$ iff $Y\in \Cond{X}$.

A sequence $(X_m)_{m\in \omega}\subseteq \Cond{\omega}$ is a fusion sequence if $\forall m>0 (X_m \preceq_m X_{m-1})$. Then we define the fusion of such a sequence to be the condition $X$ where $X(a_i)=\bigcup _m X_m(a_i)$ for every $i<n$, and for every $m\in \omega$ $X(m)=\bigcup _{j\geq m} X_j(m)$.

It is easy to check that $X\preceq_{m+1} X_m$ for every $m$.
\end{defin}

Let $X,Y\in (\omega)^\alpha_A$, where $\alpha\leq \omega$. Let us show how such a pair determines an oriented graph $(V,E)$. We set $V=A$ and $(a_i,a_j)\in E$ if $X(a_i)\cap Y(a_j)\neq \emptyset$. This motivates the following definition.
\begin{defin}
Let $G=(A,E)$ be an oriented graph with $A$ as a set of vertices. We shall always assume that $G$ contains all loops, i.e. for every $a\in A$, $(a,a)\in E$. Let $X,Y\in (\omega)^\alpha_A$, where $\alpha\leq \omega$. We say that $X$ and $Y$ are $G$-related if for every $a,b\in A$ $X(a)\cap Y(b)\neq \emptyset$ iff $(a,b)\in E$.
\end{defin}
We use similar concept to define reduced products of two copies of $(\omega)^\alpha_A$. By $\mathcal{G}$ we shall denote the set of all oriented graphs containing all loops with $A$ as the set of vertices.
\begin{defin}[Reduced products]\label{red_prod}
Let $G=(A,E)\in \mathcal{G}$. We define a reduced product $\RP{G}$ of $\Cond{\omega}\times\Cond{\omega}$ as follows: for any $(X,Y)\in \Cond{\omega}\times\Cond{\omega}$, $(X,Y)\in \RP{G}$ if whenever for some $a,b\in A$ $X(a)\cap Y(b)$, then $(a,b)\in E$. Moreover, the free blocks of $X$ and $Y$ are equal. Notice the difference from the requirement that $X$ and $Y$ are $G$-related; in particular, for any such a graph $G$ and any condition $X\in (\omega)_A^\omega$ we have $(X,X)\in \RP{G}$.

The order relation $\LP{G}$ on $\RP{G}$ is inherited from $\Cond{\omega}\times\Cond{\omega}$.

\end{defin}
\begin{proof}[Proof of Theorem \ref{CScan}]
Let $n=|A|$ and again assume that $A$ is enumerated as $\{a_0,\ldots,a_{n-1}\}$. 

The following lemma will be the main tool.
\begin{lem}\label{mainlem}
\emph{}\\
\begin{enumerate}[(i)]
\item Let $X\in \Cond{\omega}$ be any condition, $\mathcal{H}\subseteq \mathcal{G}$ any subset of the set of graphs $\mathcal{G}$ and let $M$ be a countable elementary submodel of some $H_\lambda$, where $H_\lambda$ is sufficiently large, which contains $X$ and $E$. Then there exists $Y\preceq_0 X$ such that $\forall z,y\in [Y]$ if there is some $G\in \mathcal{G}$ such that $z$ and $y$ are $G$-related  then the pair $(z,y)$ is $M$-generic for $\RP{G}$.

\item Let $G\in \mathcal{G}$ and let $(Z',Y')\LP{G} (X,X)$ be any condition and let $M$ be again a countable elementary submodel of a large enough structure containing $(Z',Y')$ and $E$. Then there exists $(Z,Y)\LP{G} (Z',Y')$ such that $\forall z\in [Z]y\in [Y]$ if $z$ and $y$ are $G$-related, then the pair $(z,y)$ is $M$-generic for $\RP{G}$.
\end{enumerate}
\end{lem}
Before we prove the lemma we show how Theorem \ref{CScan} follows.

Note that for each oriented graph $G\in \mathcal{G}$ the forcing notion $\RP{G}$ adds a pair of $G$-related elements from $(\omega)_A^0$ denoted as $x_L$ and $x_R$. By $\Vdash_G$, we mean the forcing relation related to the forcing notion $\RP{G}$.
\begin{enumerate}

\item {\bf Case 1} $\forall G\in \mathcal{G}((X,X)\Vdash_G x_L\cancel{E}x_R)$.
\item {\bf Case 2} $\exists G\in \mathcal{G}\exists (Z',Y')\LP{G} (X,X)((Z',Y')\Vdash_G x_L E x_R)$.\\

\end{enumerate}
\noindent\underline{Suppose that {\bf Case 1} holds.}\\

\noindent We fix some countable elementary submodel $M$ of a large enough structure containing everything necessary and use Lemma \ref{mainlem} (i) with this $M$, condition $X$ and $\mathcal{H}=\mathcal{G}$. We get some $Y\preceq_0 X$ such that for every $z,y\in [Y]$ if there is some $G\in \mathcal{G}$ such that $z$ and $y$ are $G$-related, then the pair $(z,y)$ is $M$-generic for $\RP{G}$. 

We claim that it follows that $E\upharpoonright [Y]=\mathrm{id}([Y])$. This is immediate. Let $z\neq y\in [Y]$ be arbitrary. Let $G\in \mathcal{G}$ be such that for every $a,b\in A$ $(a,b)\in E$ iff $z(a)\cap y(b)\neq \emptyset$ (note that WLOG we can assume that for every $a\in A$ $(a,a)\in E$ simply by requiring that for every $a\in A$ $Y(a)\cap \omega\neq \emptyset$). Then $z$ and $y$ are $G$-related. Since $(Y,Y)\Vdash_G x_L\cancel{E}x_R$ we have that $M[z,y]\models z \cancel{E} y$, where $M[z,y]$ is the $\RP{G}$-generic extension of $M$ obtained by adding the $M$-generic pair $(z,y)$. Since this is a coanalytic formula, using analytic absoluteness (Fact \ref{analabs}), we get that $z \cancel{E} y$.\\

\noindent\underline{Suppose that {\bf Case 2} holds.}\\

\noindent Let $\mathcal{G}_0\subseteq \mathcal{G}$ be the set of all graphs such that $\forall G\in \mathcal{G}_0 \exists (Z',Y')\LP{G} (X,X)((Z',Y')\Vdash_G x_L E x_R)$. From the assumption, $\mathcal{G}_0\neq \emptyset$. Let $P\subseteq n^{< n}$ be the set of all non-increasing sequences of natural numbers less than $n$ of lenght less than $n$ such that for every $g\in P$ we have $\sum_{i=0}^{|g|-1} g(i)=n$ and there is a graph $G\in \mathcal{G}_0$ such that for every $i<|g|$ there is a connected component of $g(i)$ vertices (not necessarily strongly connected; i.e. for every two vertices of that component there is an undirected path from one vertex to the other). Let us order it lexicographically. Note that for every $g\in P$ $g(0)\geq 2$. Let $h\in P$ be a maximal element in $P$ and let $H\in \mathcal{G}_0$ witness it; i.e. for every $i<|h|$ there is a connected component of $H$ containing $h(i)$ vertices. Since $h$ is maximal there is no graph $G\in \mathcal{G}_0$ having a connected component of more than $h(0)$ vertices. Similarly, among those graphs $G\in \mathcal{G}_0$ having a connected component of $h(0)$ vertices there is none having some other connected component of more than $h(1)$ vertices, etc. For every $i<|h|$ let $C_i\subseteq H$ be the corresponding connected component of $H$ having $h(i)$ vertices. We have $\mathcal{B}=\{C_i:i<|h|\}\in \mathcal{P}_A$.

We fix some countable elementary submodel $M$ of a large enough structure containing everythin necessary (including $(Z',Y')$) and use Lemma \ref{mainlem} (ii). We get some $(Z,Y)\LP{H} (Z',Y')$ such that for every $z\in [Z]$ and $y\in [Y]$ such that $z$ and $y$ are $H$-related the pair $(z,y)$ is $M$-generic for $\RP{H}$. Since $(Z,Y)\Vdash_H x_L E x_R$ we have that $M[z,y]\models z E y$ and by analytic absoluteness (Fact \ref{analabs}) we obtain $z E y$. Note that WLOG we may assume that $Z$ and $Y$ are $H$-related (by strengthening the condition if necessary).

We show that the transitivity of $E$ implies something stronger. We just need one definition before.
\begin{defin}
Let $X\in (\omega)_A^\omega$ be some condition and $x\in [X]$. Let $b(x):\omega\rightarrow A$ be such that $b(x)(m)=a$ if the $m$-th free block of $X$, where the free blocks are ordered by their minimal elements, was merged with the $a$-block in $x$.

Let $Y\preceq_0 X$ and $x\in [Y]\subseteq [X]$. Occasionally, we shall write $b_X(x)$ and $b_Y(x)$ to distinguish if we view $x$ as an element of $[Y]$ or $[X]$. However, when it is clear from the context, we shall just write $b(x)$.
\end{defin}
\begin{claim}\label{oneincl}
For every $z,y\in [Z]$ if $zE_\mathcal{B} y$ then $z E y$.
\end{claim}
\noindent\emph{Proof of the claim.} We prove that for every $z\in [Z]$ and $y\in [Y]$ such that $\forall m\in \omega \forall C\in \mathcal{B}(b(z)(m)\in C\Leftrightarrow b(y)(m)\in C)$ we have $z E y$. This suffices. To see this, let $z,y\in [Z]$ be such that $zE_\mathcal{B} y$. Let $\bar{y}\in [Y]$ be the unique element of $[Y]$ such that for every $m\in \omega$ $b(\bar{y})(m)=b(y)(m)$. Then the pair $z,\bar{y}$ satisfies the condition above, thus $z E \bar{y}$. Moreover, since $y$ and $\bar{y}$ are $H$-related, we have $y E \bar{y}$, thus from transitivity $z E y$.

Let $z\in [Z]$ and $y\in [Y]$ be a pair as described above; i.e. $\forall m\in \omega \forall C\in \mathcal{B}(b(z)(m)\in C\Leftrightarrow b(y)(m)\in C)$. We abuse the graph-theoretic terminology and by an alternating path in an oriented graph we mean the path where the orientation of edges alternatively agrees with the orientation of the path and disagrees, loops do not have orientation, thus can occur anywhere in an alternating path. We shall also assume that the orientation of the first edge agrees with that of the path. Since $C_i$, for $i<|h|$, is a connected component, for every pair $a,b\in C_i$ there exists an unoriented path from $a$ to $b$ and since $H$ contains all loops there exists an alternating path from $a$ to $b$. There exists an odd number $n_H$ and an alternating path of length $n_H$ (consisting of $n_H$ edges) between any pair of vertices from $C_i$ for every $i<|h|$. To see this, just realize that we can lengthen any alternating path by adding loops at the end. For every $i<|h|$ and every pair $a,b\in C_i$ (of not necessarily distinct elements) let $p(a,b):n_H+1\rightarrow C_i$ be such a path; i.e. $p(a,b)(0)=a$, $p(a,b)(n_H)=b$, for every $m\leq n_H$ $p(a,b)(m)\in C_i$ and for every even $m<n_H$ we have $(p(a,b)(m),p(a,b)(m+1))\in E_H$ (similarly, for every odd $m<n_H$ we have $(p(a,b)(m+1),p(a,b)(m))\in E_H$), where $E_H$ is the set of edges of $H$.

We define a sequence $x_0,\ldots,x_{n_H}$ such that $x_0=z$, $x_{n_h}=y$, for odd $i$ $x_i\in [Y]$, for even $i$ $x_i\in [Z]$ and for every $i<n_H$ $x_i E x_{i+1}$. Then we get from transitivity of $E$ that $z E y$. Let us define the sets $O_i=\{m\in \omega: b(z)(m)\in C_i\}$ for every $i<|h|$. By the assumption $O_i=\{m\in \omega: b(y)(m)\in C_i\}$. For every $i<|h|$ and $m\in O_i$ let $p(m)=p(b(z)(m),b(y)(m))$. For every $i\leq n_H$ let $x_i$ be the unique element (of $[Z]$ provided $i$ is even and of $[Y]$ provided $i$ is odd) such that for every $j<|h|$ and $m\in O_j$ $b(x_i)(m)=p(m)(i)$. To check that for every even $i<n_H$ we have $x_i E x_{i+1}$ it suffices to check that $x_i$ and $x_{i+1}$ are $H$-related. That follows from the fact that $Z$ and $Y$ are $H$-related and that for every $j<|h|$ and $m\in O_j$ $b(x_i)(m)=p(m)(i)$, $b(x_{i+1})(m)=p(m)(i+1)$, thus $(b(x_i)(m),b(x_{i+1})(m))\in E_H$. To check that for every odd $0<i\leq n_H$ we have $x_{i-1} E x_i$ it again suffices to prove that $x_{i-1}$ and $x_i$ are $H$-related. This is completely analogous as for $i$ even. This finishes the proof of the claim.
\begin{claim}
For every $G\in \mathcal{G}\setminus \mathcal{G}_0$ $(Z,Z)\Vdash_G x_L \cancel{E} x_R$.
\end{claim}
Otherwise we would have that $\exists(Z',Y')\LP{G} (Z,Z)((Z',Y')\Vdash_G x_L E x_R)$; however $(Z',Y')\LP{G} (Z,Z)\LP{G} (X,X)$ and that would be a contradiction with the definition of $\mathcal{G}_0$ and the fact that $G\notin \mathcal{G}_0$.\\

Let us use Lemma \ref{mainlem} (i) with the model $M$, condition $Z$ and $\mathcal{H}=\mathcal{G}\setminus\mathcal{G}_0$. We get some $Y'\preceq_0 Z$ such that for every $G\in \mathcal{G}\setminus\mathcal{G}_0$ and every pair $z,y\in [Y']$ that is $G$-related it is also $M$-generic for $\RP{G}$. Since $(Y',Y')\Vdash_G x_L \cancel{E} x_R$, arguing as before we get that $z \cancel{E} y$. Finally, let us slightly strengthen the condition $Y'$ as follows. We define $Y\preceq_0 Y'$ as follows: for every $a\in A$ $Y(a)=Y'(a)$, however for every $m\in \omega$, the $m$-th free block of $Y$ (recall that free blocks are ordered by their minimal elements) is obtaining by merging $n^2+1$ free blocks of $Y'$ together; more specifically, let the $m$-the free block of $Y$ be the union of the $i$-th blocks of $Y'$, where $i$ ranges between $(n^2+1)\cdot m$ and $(n^2+1)\cdot m+n^2$. We recall that $n=|A|$.
\begin{claim}
$E\upharpoonright [Y]=E_\mathcal{B}\upharpoonright [Y]$.
\end{claim} 
\noindent\emph{Proof of the claim.} In Claim \ref{oneincl} we proved that $E\upharpoonright [Z]\supseteq E_\mathcal{B}\upharpoonright [Z]$ and since $Y\preceq_0 Z$, also $E\upharpoonright [Y]\supseteq E_\mathcal{B}\upharpoonright [Y]$. Suppose for contradiction that there is a pair $z,y\in [Y]$ such that $z\cancel{E}_\mathcal{B} y$, yet $z E y$. Thus there exist $i_0<i_1<|h|$ and $a\in C_{i_0}$ and $b\in C_{i_1}$ such that $z(a)\cap y(b)\neq \emptyset$. Thus there is some $m\in \omega$ such that $b(z)(m)=a$ while $b(y)(m)=b$. Since also $z\in [Y']$ we have $b_{Y'}(z)(i)=a$ for $m\cdot (n^2+1)\leq i<(m+1)\cdot (n^2+1)$. We define $y'\in [Y']$ as follows: for every $i\notin [m\cdot (n^2+1),(m+1)\cdot (n^2+1)-1]$ we set $b_{Y'}(y')(i)=b_{Y'}(z)(i)$, we also set $b_{Y'}(y')(i)=b_{Y'}(z)(i)=a$ for $i=m\cdot (n^2+1)$, and finally we set $b_{Y'}(y')(i)$, for $i\in (m\cdot (n^2+1),(m+1)\cdot (n^2+1))$ so that $z$ and $y'$ are $H$-related. This is clearly possible. We have that $z E y'$ and from transitivity of $E$ also $y' E y$. However, consider the graph $G=(A,E_G)$ determined by the pair $(y',y)$; i.e. $(a_i,a_j)\in E_G$ if $y'(a_i)\cap y(a_j)\neq \emptyset$. $b_{Y'}(y')(i)$ and $b_{Y'}(y)(i)$, for $i\in [m\cdot (n^2+1),(m+1)\cdot (n^2+1)-1]$, witness that the connected components of $G$ are $\{C_i:i<|h|\wedge i\neq i_0\wedge i\neq i_1\}\cup\{C_{i_o}\cup C_{i_1}\}$. It follows that $G\in \mathcal{G}\setminus \mathcal{G}_0$. Otherwise, it would contradict the choice of $h$ as a maximal element of $P$. Therefore, since $y'$ and $y$ are $G$-related, we have $y'\cancel{E} y$, a contradiction. This finishes the proof of the claim.\\

To finish the proof of Theorem \ref{CScan} it remains to prove Lemma \ref{mainlem}.\\

\noindent\emph{Proof of Lemma \ref{mainlem}.} We prove only the item (i), the second one is just a routine modification. What we prove is that for every elementary submodel $M$ of a large enough structure, every condition $X\in (\omega)_A^\omega$ and $G\in \mathcal{G}$ there is $Y\preceq_0 X$ such that for every $z,y\in [Y]$ if the pair $z$ and $y$ is $G$-related then this pair is $M$-generic for $\RP{G}$. Repeating this claim for every $G\in \mathcal{H}$ gives the item (i) of the lemma.

So let an elementary submodel $M$ of $H_\lambda$, $\lambda$ suficiently large, a graph $G=(A,E_G)\in \mathcal{G}$ and a condition $X\in (\omega)_A^\omega$ be given. Let us enumerate all open dense subsets of $\RP{G}$ lying in $M$ as $(D_n)_{n\in \omega}$.

Let $E'=E_G\setminus\{(a,a):a\in A\}$ be the set of all edges of $G$ that are not loops. Let $e=|E'|$ and let $u,v\in n^e$ be a pair of sequences of length $e$ such that $\{(a_{u(i)},a_{v(i)}):i<e\}=E'$.Certainly we can find $e!$-many of such sequences, let us enumerate them as $(u_i,v_i)_{i<e!}$; i.e. for every $i\neq j<e!$, $u_i,v_i\in n^e$, $u_i\neq u_j\vee v_i\neq v_j$ and for every $j<e!$ $\{(a_{u_j(i)},a_{v_j(i)}):i<e\}=E'$.
\begin{claim}\label{mainlemingr}
Let $Z\preceq_0 X$ be arbitrary, $i<e!$ and for every $j<e$ let $w_j\in n^{<\omega}$. Then there exists $T\preceq_{e+|w_0|+\ldots+|w_{e_1}|} Z$ such that for every $z\in [T^{w_0u_i(0)w_1u_i(1)\ldots w_{e-1}u_i(e-1)}]$, every $y\in [T^{w_0v_i(0)w_1v_i(1)\ldots w_{e-1}v_i(e-1)}]$ such that $z$ and $y$ are $G$-related, the pair $(z,y)$ is $M$-generic for $\RP{G}$.
\end{claim}
Suppose the claim is proved. We show how, by a fusion process, we can find the desired $Y$. Using the claim $e!$ times we can find $Y_0\preceq_{e} X$ such that for every $i<e!$ and for every $z\in [Y_0^{u_i(0)\ldots u_i(e-1)}]$ and every $y\in [Y_0^{v_i(0)\ldots v_i(e-1)}]$ if $z$ and $y$ are $G$-related then they are $M$-generic for $\RP{G}$. Suppose we have already found $Y_{m-1}\preceq_{e+m-1} Y_{m-2}$. Using the claim several times we can find $Y_m\preceq_{e+m} Y_{m-1}$ such that for every $i<e!$ and every $e$-tuple of sequences $w_0,\ldots,w_{e-1}\in n^{<\omega}$ (each $w_i$ can be empty) such that $\sum _{i=0}^{e-1} |w_i|=m$ we have that for every $z\in [Y_m^{w_0u_i(0)w_1u_i(1)\ldots w_{e-1}u_i(e-1)}]$ and every $y\in [Y_m^{w_0v_i(0)w_1v_i(1)\ldots w_{e-1}v_i(e-1)}]$ if $z$ and $y$ are $G$-related, then this pair is $M$-generic for $\RP{G}$. Let $Y$ be the limit of the fusion sequence $Y_0\succeq_{e+1} Y_1\succeq_{e+2}\ldots$. Let $z,y\in [Y]$ be a pair that is $G$-related. Realize that it is witnessed by some $i<e!$ and some $e$-sequence $w_0,\ldots,w_{e-1}\in n^{<\omega}$ (each $w_i$ can be empty) such that $z\in [Y^{w_0u_i(0)w_1u_i(1)\ldots w_{e-1}u_i(e-1)}]$ and $y\in [Y^{w_0v_i(0)w_1v_i(1)\ldots w_{e-1}v_i(e-1)}]$. Let $m=\sum _{i=0}^{e-1} |w_i|$. However, we then have that $z\in [Y_m^{w_0u_i(0)w_1u_i(1)\ldots w_{e-1}u_i(e-1)}]$ and $y\in [Y_m^{w_0v_i(0)w_1v_i(1)\ldots w_{e-1}v_i(e-1)}]$ and since $z$ and $y$ are $G$-related we have guaranteed at the $m$-th step of the fusion that the pair $(z,y)$ is $M$-generic for $\RP{G}$.

Thus it remains to prove the claim.\\

Let $s,t\in n^{<\omega}$ be a pair of sequences of the same length. We say that such a pair $(s,t)$ is good if for every $i<|s|=|t|$ $(a_{s(i)},a_{t(i)})\in E_G$ (it may be a loop, i.e. $s(i)$ and $t(i)$ may be equal). For every $m\in \omega$ let $(s^m_i,t^m_i)_{i<|E_G|^m}$ be an enumeration of all good pairs of sequences of length $m$. We shall again do a fusion. Since $D_0$ is dense open in $\RP{G}$ there exists $T_0\preceq_{e+|w_0|+\ldots+|w_{e_1}|} Z$ such that $(T_0^{w_0u_i(0)\ldots w_{e-1}u_i(e-1)},T_0^{w_0v_i(0)\ldots w_{e-1}v_i(e-1)})\in D_0$. Assume we have found $T_{m-1}\preceq_{e+|w_0|+\ldots+|w_{e_1}|+m-1} T_{m-2}$. Since $D_m$ is dense open in $\RP{G}$ we can find $T_m\preceq_{e+|w_0|+\ldots+|w_{e_1}|+m} T_{m-1}$ such that for every $j<|E_G|^m$ $(T_m^{w_0u_i(0)\ldots w_{e-1}u_i(e-1)s^m_j},T_m^{w_0v_i(0)\ldots w_{e-1}v_i(e-1)t^m_j})\in D_m$. Let $T$ be the limit of this fusion sequence. We claim it is as desired. Let $z\in [T^{w_0u_i(0)w_1u_i(1)\ldots w_{e-1}u_i(e-1)}]$ and $y\in [T^{w_0v_i(0)w_1v_i(1)\ldots w_{e-1}v_i(e-1)}]$ be such that they are $G$-related. Then for every $m$ there is $j<|E_G|^m$ such that $z\in [T_m^{w_0u_i(0)\ldots w_{e-1}u_i(e-1)s^m_j}]$ and $y\in [T_m^{w_0v_i(0)\ldots w_{e-1}v_i(e-1)t^m_j}]$, thus for every $m\in \omega$ $(z,y)\in \bigcup D_m$, so the pair $(z,y)$ is $M$-generic for $\RP{G}$.
\end{proof}
\bigskip

We do not know whether an analogous result holds for partitions of $\omega$ having free blocks. Let us consider the following ideal $I_{C_{n,m}}$, where $n\in \omega$ and $m\geq 1$, on the space $\Condi{\omega}{m}$: a Borel subset $A\subseteq \Condi{\omega}{m}$ belongs to $I_{C_{n,m}}$ if there is no $X\in \Cond{\omega}$ such that $\Condi{X}{m}\subseteq A$. For any partition $\mathcal{B}$ of $A\cup m$ let us define an equivalence relation $E_\mathcal{B}$ on $\Condi{\omega}{m}$: for $x,y\in \Condi{\omega}{m}$ we set $x E_\mathcal{B} y$ iff $\forall n\in \omega\forall B\in \mathcal{B} (x(n)\in B\Leftrightarrow y(n)\in B)$. The following is our hypothesis concerning canonization for the ideal $I_{C_{n,m}}$.
\begin{question}
Let $X\in \Cond{\omega}$ be a condition and $E$ an analytic equivalence relation on $\Condi{X}{m}$. Does there exist a partition $\mathcal{B}$ of $A\cup m$ and $Y\in \Cond{X}$ such that $E\upharpoonright \Condi{Y}{m}=E_\mathcal{B}\upharpoonright \Condi{Y}{m}$?
\end{question}
A result of this type would be a generalization of the full Carlson-Simpson theorem.\\

We finish by providing the proofs of Proposition \ref{CSprop} and Theorem \ref{introcor}.
\begin{proof}[Proof of Proposition \ref{CSprop}]
Fix an alphabet $A$ with $|A|=n\geq 2$. Let us prove (1). Let $A_n\in I_{C_n}$ for all $n\in \omega$. Suppose that $A=\bigcup_{n\in \omega} A_n \notin I_{C_n}$. It must contain $[X]$ for some $X\in \Cond{\omega}$. We can define a Borel equivalence relation $E$ on $[X]$ with countably many classes such that $\forall x,y\in [X] (xEy\Leftrightarrow \exists n\in \omega (x,y\in A_n))$. Applying Theorem \ref{CScan} we get $Y\preceq_0 X$ such that $E\upharpoonright [Y]$ is the full relation, the identity relation or $E_\mathcal{B}\upharpoonright [Y]$ for some $\mathcal{B}\in \mathcal{P}_A$. Since $E$ has only countably many classes only the first case is possible. Thus $E\upharpoonright [Y]=[Y]\times[Y]$, i.e. there is $n\in \omega$ such that $[Y]\subseteq A_n$ which is a contradiction.

The item (2) follows from (1). For any $X\in \Cond{\omega}$, $[X]$ is a Borel (closed) $I_{C_n}$-positive subset; conversely, it follows from (1) that any Borel $I_{C_n}$-positive subset of $A^\omega$ contains $[X]$ for some $X\in \Cond{\omega}$.

We now prove (3). Consider the suborders $\preceq_m\subseteq \preceq_{m-1}\subseteq \ldots \subseteq \preceq_0$ on $\Cond{\omega}$. It is easy to check that $\Cond{\omega}$ with these relations satisfies Axiom A and thus it is proper (we refer to \cite{Je} p. 604, for example).
\end{proof}
\begin{proof}[Proof of Theorem \ref{introcor}]
This is just a special case of Theorem \ref{CScan} for $A=\{a_0,a_1\}$ and if we consider $X$ to be the biggest condition in $\Cond{\omega}$. Theorem \ref{CScan} gives a subcondition $Y\in \Cond{\omega}$ on which $E$ is simple. The condition $Y$ determines the sequence $(A_n)_{n\in \omega}$: $A_0=Y(a_1)$ and $A_i=A(i-1)$ for $i\geq 1$.

Let us just note that we cannot eliminate the set $A_0$ from the statement, i.e. demand it to be empty. Just consider an equivalence relation $E$ on $\mathcal{P}(\omega)$ where for $a,b\in \mathcal{P}(\omega)$ we have $a E b$ if $\min a=\min b$.
\end{proof}


\begin{thebibliography}{10}
\bibitem{CaSi}
T. Carlson, S. Simpson, \emph{A dual form of Ramsey's theorem}, Adv. in Math.  53 (1984), no. 3, 265–290
\bibitem{Je}
T. Jech, \emph{Set theory. The third millennium edition}, Springer-Verlag, Berlin, 2003
\bibitem{KSZ}
V. Kanovei, M. Sabok, J. Zapletal, \emph{Canonical Ramsey Theory on Polish Spaces}, Cambridge University Press, in press
\bibitem{Zap}
J. Zapletal, \emph{Forcing idealized}, Cambridge University Press, Cambridge, 2008
\end{thebibliography}
\end{document}